\documentclass[leqno,final]{siamltex}

\usepackage{amsmath}
\usepackage{amssymb}
\usepackage{graphicx}
\usepackage{epstopdf}
\usepackage{color}

\newcommand{\monic}[3]{{P_s(#1;#2)\,#3}}
\newcommand{\mpol}[2]{{P_s(#1;#2)}}
\newcommand{\monicQ}[3]{{Q_{2s}(#1;#2)\,#3}}
\newcommand{\Qpol}[2]{{Q_{2s}(#1;#2)}}

\newcommand{\R}{\mathbb{R}}
\newcommand{\K}{\mathcal{K}}
\newcommand{\M}{\mathcal{M}}

\newcommand{\tv}{\widetilde{v}}
\newcommand{\tw}{\widetilde{w}}

\newtheorem{example}[theorem]{Example}

\title{On the Forsythe conjecture\footnote{Version of September 29, 2022}}
\author{Vance Faber\footnotemark[2], J\"org Liesen\footnotemark[3] and Petr Tich\'y\footnotemark[4]}

\begin{document}

\maketitle

\renewcommand{\thefootnote}{\fnsymbol{footnote}}
\footnotetext[2]{
Hoquiam, WA 98550, USA. Email: {\tt vance.faber@gmail.com}}
\footnotetext[3]{Institute of Mathematics, TU Berlin, Stra{\ss}e des 17.~Juni 136, 10623 Berlin, Germany.  Email: {\tt liesen@math.tu-berlin.de}.}
\footnotetext[4]{
Faculty of Mathematics and Physics, Charles University, Sokolovsk\'a 83, 18675 Prague, Czech Republic. Email: {\tt petr.tichy@mff.cuni.cz}. The work of Petr~Tich\'{y} was supported by the Grant Agency of the Czech Republic under grant no. 20-01074S.
}

\begin{abstract}
Forsythe formulated a conjecture about the asymptotic behavior of the restarted conjugate gradient method in 1968. We translate several of his results into modern terms, and generalize the conjecture (originally formulated only for symmetric positive definite matrices) to symmetric and nonsymmetric matrices. Our generalization is based on a two-sided or cross iteration with the given matrix and its transpose, which is based on the projection process used in the Arnoldi (or for symmetric matrices the Lanczos) algorithm. We prove several new results about the limiting behavior of this iteration, but the conjecture still remains largely open.
\end{abstract}



\section{Introduction}
In his paper~\cite{For68} from 1968, Forsythe considered the ``asymptotic directions of the $s$-dimensional optimum gradient method'' or, in modern terms, the asymptotic behavior of the conjugate gradient (CG) method~\cite{HesSti52} when restarted every~$s$ steps. Based on what he could prove in the simplest case of restart length $s=1$, and what he observed numerically for $s=2$, Forsythe formulated a conjecture about the behavior of the method for general~$s$. In short, the Forsythe conjecture is that the normalized error vectors (or normalized residual vectors) of the restarted method eventually cycle back and forth between two limiting directions. This unwanted oscillatory behavior leads to slow (at most linear) convergence, which Stiefel had illustrated with his cage (``K\"afig'') already in 1952~\cite{Sti52}.

In light of the importance and wide-spread use of the CG method, surprisingly few results about the Forsythe conjecture have been published beyond the case $s=1$, and the conjecture remains largely open after more than 50 years. In addition to Forsythe in~\cite{For68}, proofs of the oscillatory behavior for $s=1$ were published by Akaike~\cite{Aka59} and Gonzaga and Schneider~\cite{GonSch16}, as well as Afanasjew, Eiermann, Ernst, and G\"uttel~\cite{AfaEieErnGue08} (for a somewhat different algorithm). The nature of the oscillations for $s=1$ were studied by Nocedal, Sartenaer, and Zhu~\cite{NocSarZhu02}. Prozanto, Wynn, and Zhigljavsky~\cite{ProWynZhi09} analyzed the rate of convergence of the $s$-dimensional optimum gradient method, re-derived several results of Forsythe in a different formulation, and numerically studied the limiting behavior particularly in the case $s=2$. This case was also analyzed by Zhuk and Bondarenko~\cite{ZhuBon83}, and we will comment on their work after the proof of our Theorem~\ref{thm:limit-degree} below.

Forsythe already suspected, and even conjectured, that the results he proved apply to a wider class of (restarted) gradient-type methods when applied to sufficiently smooth functions; see~\cite[p.~58]{For68}. And indeed, a cyclic behavior has sometimes been observed also in other restarted Krylov subspace methods, particularly when applied to symmetric or normal matrices; see, e.g.,~\cite{BakJesMan05,VecLan10}.

In this paper we generalize the Forsythe conjecture (formulated only for symmetric positive definite matrices) to symmetric and nonsymmetric matrices. Instead of the projection process underlying the CG method, we use the projection process of the Arnoldi (or for symmetric matrices the Lanczos) algorithm. The resulting \emph{Arnoldi cross iteration} alternatingly uses the given matrix $A$ and its transpose~$A^T$. While such a cross iteration usually does not solve a linear algebraic system or eigenvalue problem for~$A$, it can be of interest for analyzing the worst-case behavior of the underlying projection process. Motivated by work of Zavorin~\cite{Zav01} as well as Zavorin, O'Leary, and Elman~\cite{ZavOLeElm03} (who studied complete stagnation of the GMRES method~\cite{SaaSch86} for~$A$ and~$A^T$), we have previously introduced a \emph{GMRES cross iteration} in the context of worst-case GMRES~\cite{FabLieTic13}. In this paper we do not focus on the worst-case behavior, but we will nevertheless show that the Arnoldi cross iteration with $s=1$ solves the \emph{ideal Arnoldi problem} for real orthogonal matrices; see~\eqref{eqn:idealArnoldi} and the end of Section~\ref{sec:ACI-orth}.
Because it uses $A$ and $A^T$, the Arnoldi cross iteration for $s=1$ is related to Ostrowski's two-sided iteration~\cite{Ost59} and Parlett's Alternating Rayleigh Quotient Iteration (ARQI) for computing eigenpairs of nonsymmetric matrices~\cite{Par74}. The precise nature of these relations are yet to be explored.

For symmetric matrices the iterative steps with $A$ and~$A^T$ in the Arnoldi cross iteration coincide, and we recover essentially the same algorithm as originally studied by Forsythe, only with a different (namely the Arnoldi or the Lanczos) projection process. We conjecture that the Arnoldi cross iteration in general has the same oscillatory limiting behavior as the one conjectured by Forsythe for the restarted CG method. We point out that usually such a behavior does not exist in restarted Krylov subspace method for nonsymmetric matrices which use only $A$ (and not $A^T$); see~\cite{VecLan11}.

We generalize or extend several results from the original Forsythe formulation and from our paper~\cite{FabLieTic13} to the Arnoldi cross iteration. We also prove a few new results about the limiting behavior of this iteration. In the simplest case $s=1$ and $A=A^T$, the iteration reduces to the one studied in~\cite{AfaEieErnGue08}, and we give an alternative (and in particular simpler) proof of the limiting behavior in this case. Another new result in this paper is the proof of the limiting behavior of the Arnoldi cross iteration for $s=1$ and orthogonal matrices with eigenvalues having only positive (or only negative) real parts.

Besides being more general than the original formulation of Forsythe, we hope that our translation of the theory in the context of the conjecture into more modern notation based on projection processes is more transparent, and therefore will motivate further research beyond this paper.

The paper is organized as follows. As mentioned above, in Section~\ref{sec:forsythe} we present the original Forsythe conjecture in modern terms and discuss previous attempts to prove it. In Section~\ref{sec:arnoldi} we introduce the Arnoldi cross iteration, prove several results about this iteration and its limiting behavior, and generalize the conjecture. In Section~\ref{sec:special} we consider the Arnoldi cross iteration for symmetric and orthogonal matrices. In Section~\ref{sec:Conclusions} we give concluding remarks.

\smallskip{}
\emph{Notation.} Throughout the paper we consider only real matrices for notational simplicity. Many results can be easily extended to the complex case. The degree of the minimal polynomial of a matrix $A\in {\mathbb R}^{n\times n}$ is denoted by $d(A)$, and the grade of a vector $v\in{\mathbb R}^{n}$ with respect to $A$ by $d(A,v)$; cf.~\cite[Definition~4.2.1]{LieStr13}. For each $k\geq 1$, the $k$th Krylov subspace of $A\in\R^{n\times n}$ and $v\in\R^n$ is denoted by $\mathcal{K}_{k}(A,v):={\rm span} \{v,Av,\dots,A^{k-1}v\}$. The Euclidean norm on $\R^n$ is denoted by $\|\cdot\|$, and for a symmetric positive definite matrix $A\in \R^{n \times n}$ the $A$-norm on $\R^n$ is denoted by $\|\cdot\|_A$, where $\|v\|_A:=(v^TAv)^{1/2}$.

\section{The original Forsythe conjecture}\label{sec:forsythe}

Forsythe~\cite{For68} considered minimizing functions of the form
\begin{equation}\label{eqn:f}
f(x)=\frac12 x^TAx-x^Tb,
\end{equation}
where $A\in {\mathbb R}^{n\times n}$ is symmetric positive definite, by an iterative method he called the \emph{optimum $s$-gradient method}. We remark that Forsythe actually considered $b=0$ in~\eqref{eqn:f} for simplicity, so that the unique minimizer of~$f$ is given by $x=0$. In order to avoid any possible confusion of the reader that may be caused by studying a method for computing a solution that is already known trivially, we here consider a possibly nonzero vector~$b\in\R^n$ in~\eqref{eqn:f}. In any case, the unique minimizer of~$f$ is equal to the uniquely determined solution of the linear algebraic system $Ax=b$.

We now briefly describe Forsythe's optimum $s$-gradient method using modern notation. Let $x_0\in {\mathbb R}^n$ be a given initial vector. Then the gradient of $f$ evaluated at $x_0$ is given by
$$\nabla f(x_0)=Ax_0-b=:-r_0,$$
and the optimum 1-gradient method amounts to finding the minimum of $f$ with respect to the $A$-norm on the line $x_0-\alpha\nabla f(x_0)=x_0+\alpha r_0$, which is the \emph{line of steepest descent} of $f$ at $x_0$. Thus, using the notation of Krylov subspaces, the new iterate of the optimum 1-gradient method satisfies
\begin{equation}\label{eqn:CG1}
x_1\in x_0+{\rm span}\{r_0\}=x_0+\K_1(A,r_0),
\end{equation}
and it is found so that
$$f(x_1) = \min_{z\in x_0+\K_1(A,r_0)} f(z).$$
Since $\|x-x_1\|_A^2=2 f(x_1)+x^TAx$, solving the minimization problem for~$f$ is equivalent with finding~$x_1$ so that
$$\|x-x_1\|_A=\min_{z\in x_0+\K_1(A,r_0)}\|x-z\|_A.$$
It is well known (see, e.g.,~\cite[Theorem~2.3.1]{LieStr13}) that~$x_1$ is uniquely determined by the orthogonality condition
\begin{equation}\label{eqn:CG2}
x-x_1\perp_A \K_1(A,r_0).
\end{equation}
The equations \eqref{eqn:CG1}--\eqref{eqn:CG2} give a complete mathematical characterization of the optimum 1-gradient method.

For each $z=x_0+\alpha r_0\in x_0+\K_1(A,r_0)$ we have
$$\nabla f(z)= A(x_0+\alpha r_0)-b=-r_0+\alpha Ar_0.$$
Forsythe called the set
$$\{x_0+\alpha_1 r_0+\alpha_2 Ar_0\,:\,\alpha_1,\alpha_2\in {\mathbb R}\}
= x_0+\K_2(A,r_0)$$
the \emph{$2$-plane of steepest descent} of $f$ at $x_0$. Analogously, the \emph{$s$-dimensional plane of steepest descent} of $f$ at $x_0$ is given by
$$\Big\{x_0+\sum_{j=0}^{s-1} \alpha_j A^jr_0\,:\,\alpha_j\in {\mathbb R},\,j=0,1,\dots,s-1\Big\}
= x_0+\K_s(A,r_0).$$
In analogy with \eqref{eqn:CG1}--\eqref{eqn:CG2}, the iterate $x_1$ of the optimum $s$-gradient method is then defined by the relations
\begin{equation}\label{eqn:forsythe}
x_1\in x_0+\K_s(A,r_0)\quad\mbox{such that}\quad x-x_1\perp_A \K_s(A,r_0).
\end{equation}
This is nothing but the mathematical characterization of the $s$th iterate of the CG method applied to $Ax=b$ with the initial vector $x_0$; see, e.g.,~\cite[Theorem~2.3.1]{LieStr13}. Forsythe was of course aware that the method he considered is mathematically equivalent to the CG method. He also pointed out that the implementation of Hestenes and Stiefel ``in practice ... may usually be superior to the optimum $s$-gradient methods''~\cite[p.~58]{For68}.

Forsythe was interested in the behavior of the iterates in the optimum $s$-gradient method when it is applied multiple times or, in modern terms, \emph{restarted}. It is well known that the restarted method converges to the uniquely determined minimizer of~$f$, given by $x=A^{-1}b$. The interesting question in this context is \emph{from which directions} the iterates approach their limit.

In order to study this behavior, one considers an integer~$s$ with $1\leq s< d(A)$, an initial vector $x_0\in {\mathbb R}^{n}$ with $d(A,x_0)\geq s+1$, and a sequence of vectors constructed using \eqref{eqn:forsythe} with an additional normalization:
\begin{align}
& \mbox{For \ensuremath{k=0,1,2,\dots}}\nonumber\\
& \hspace{1cm}\mbox{$y_k=r_k/\|r_k\|$,}\label{eqn:yk}\\
& \hspace{1cm}\mbox{$x_{k+1}\in x_k+\K_s(A,r_k)$
such that $x-x_{k+1}\perp_A \K_s(A,r_k)$.}\label{eqn:forsythe1}
\end{align}

For the case $s=1$, and hence the steepest descent method, Forsythe and Motzkin had conjectured already in 1951 that the two sequences of normalized vectors with even and odd indices, i.e., $\{y_{2k}\}$ and $\{y_{2k+1}\}$, alternate asymptotically between two limit vectors that are determined by the eigenvectors of $A$ corresponding to its smallest and its largest eigenvalue. In the words of Forsythe, ``the iteration behaves asymptotically, as $k\rightarrow\infty$, as though it were entirely in the two-space $\pi_{1,n}$''~\cite[p.~64]{For68}. Since asymptotically the vectors with the even indices become arbitrarily close to being collinear, and the same happens for the vectors with the odd indices, the error norms of the non-normalized restarted iteration can converge to zero (only) linearly; see~\cite[pp.~63--64]{For68} for Forsythe's original discussion of this important observation.

The conjecture of Forsythe and Motzkin (for $s=1$) was proven by Akaike in 1959 using methods from probability theory~\cite{Aka59}, and Forsythe gave another proof in~\cite{For68} using orthogonal polynomials. Based on numerical evidence he suspected that the behavior of the optimum $s$-gradient method is similar for all $s$, and we can therefore state his conjecture from~\cite[p.~66]{For68} as follows:

\medskip
{\sc Forsythe conjecture.} \emph{For $2\leq s<d(A)$, each of the two subsequences $\{y_{2k}\}$ and $\{y_{2k+1}\}$ in \eqref{eqn:yk}--\eqref{eqn:forsythe1} has a single limit vector.}

\section{The Arnoldi cross iteration}\label{sec:arnoldi}

Since the (oblique) projection process~\eqref{eqn:forsythe1} on which the Forsythe conjecture is based uses the $A$-inner product, the matrix $A$ must be symmetric positive definite. Our goal in this section is to obtain a generalization of the Forsythe conjecture to symmetric and nonsymmetric matrices that is as straightforward as possible. We therefore consider a closely related projection process, which is also well known in the area of Krylov subspace methods, and which is well defined for general~$A$.

Given $A\in\mathbb{R}^{n\times n}$, an integer $s\geq1$, and a vector $v\in\mathbb{R}^{n}$, we define the vector $w\in\mathbb{R}^{n}$ such that
\begin{equation}\label{eqn:arnoldi}
w\in A^{s}v+\mathcal{K}_{s}(A,v)\quad\mbox{and}\quad w\perp\mathcal{K}_{s}(A,v).
\end{equation}
The construction of the vector $w$ is the $s$th step in the orthogonalization of the Krylov sequence $v,Av,A^{2}v,\dots$, where the vector $A^{s}v$ is orthogonalized with respect to the previous vectors $v,\dots,A^{s-1}v$. Since the standard method for computing orthogonal Krylov subspace bases is the Arnoldi algorithm~\cite{Arn51}, we call $w$ the \emph{Arnoldi projection} of~$v$ with respect to $A$ and $s$.

By construction, $w=p(A)v$ for some polynomial $p\in\mathcal{M}_{s}$, which is the set of (real) monic polynomials of degree~$s$. Further basic properties of the Arnoldi projection are shown in the next result.
\medskip{}

\begin{lemma}\label{lem:basic}
If $A\in\mathbb{R}^{n\times n}$, $1\leq s<d(A)$, $v\in\mathbb{R}^{n}$, and $w$ is given by \eqref{eqn:arnoldi}, then the following
hold:
\begin{itemize}
\item[(1)] The vector $w$ satisfies
\[
\|w\|=\min_{z\in\mathcal{K}_{s}(A,v)}\|A^{s}v+z\|=\min_{q\in\mathcal{M}_{s}}\|q(A)v\|,
\]
and hence $w=0$ if and only if $d(A,v)\leq s$.
\item[(2)] If $d(A,v)\geq s$, then $w=\monic{A}{v}{v}$ for a uniquely determined polynomial
$\mpol{z}{v} \in\mathcal{M}_{s}$.
\end{itemize}
\end{lemma}

\begin{proof}
(1) By construction, $w=A^{s}v+u$ for some $u\in\mathcal{K}_{s}(A,v)$,
and hence for any $z\in\mathcal{K}_{s}(A,v)$ we obtain
\begin{align*}
\|A^{s}v+z\|^{2}=\|w-(u-z)\|^{2}=\|w\|^{2}+\|u-z\|^{2}\geq\|w\|^{2},
\end{align*}
with equality if and only if $z=u$. This proves the minimization
property of $w$. From $\|w\|=\min_{q\in\mathcal{M}_{s}}\|q(A)v\|$
we see that $w=0$ holds if and only if $d(A,v)\leq s$.

(2) Suppose that $d(A,v)\geq s$. If $w=p(A)v=q(A)v$ with $p,q\in\mathcal{M}_{s}$,
then $(p-q)(A)v=0$. Since the polynomials $p$ and $q$ are both
monic, the polynomial $p-q$ has degree at most $s-1$. But then $d(A,v)\geq s$
implies $p-q=0$.
\end{proof}
\medskip{}

In order to generalize the Forsythe conjecture to nonsymmetric matrices we will focus on the limiting behavior of a sequence of vectors obtained from repeatedly computing~\eqref{eqn:arnoldi} alternatingly with $A$ and $A^T$, and with additional normalizations.
The algorithm we consider here is similar to~\cite[Algorithm~1]{FabLieTic13}, but based on the Arnoldi projection~\eqref{eqn:arnoldi} instead of the projection process that is used in the GMRES method.

Given $A\in\mathbb{R}^{n\times n}$, an integer $s$ with $1\leq s<d(A)$, and a vector $v_{0}\in\mathbb{R}^{n}$ with $\|v_0\|=1$ and $d(A,v_0)\geq s+1$, we consider the following algorithm:
\begin{align}
 & \mbox{For \ensuremath{k=0,1,2,\dots}}\nonumber \\
 & \hspace{1cm}\mbox{\ensuremath{\tw_{k}=\monic{A}{v_k}{v_k}},}\label{eqn:twk}\\
 & \hspace{1cm}\mbox{\ensuremath{w_{k}=\tw_{k}/\|\tw_{k}\|},}\label{eqn:wk}\\
 & \hspace{1cm}\mbox{\ensuremath{\tv_{k+1}=\monic{A^T}{w_k}{w_k}},}\label{eqn:tvk}\\
 & \hspace{1cm}\mbox{\ensuremath{v_{k+1}=\tv_{k+1}/\|\tv_{k+1}\|}.}\label{eqn:vkp1}
\end{align}
We call this algorithm the \emph{Arnoldi cross iteration with restart length~$s$}, or shortly \emph{ACI($s$)}.
Because of the mathematical similarities between the ACI($s$) and the GMRES cross iteration in~\cite{FabLieTic13}, we expect that their limiting behavior is similar.

We will now transfer some of the results about the GMRES cross iteration in~\cite[Section~2]{FabLieTic13} to the ACI($s$). That paper is about worst-case GMRES, and here we need an analogous definition for the worst case in the projection process \eqref{eqn:arnoldi}.
\medskip{}

\begin{definition}
Let $A\in\mathbb{R}^{n\times n}$ and an integer $s$ with $1\leq s<d(A)$ be given.
Denote
\begin{equation}\label{eq:wcArnoldi}
\Phi_s(A):=\max_{\substack{v\in\R^n\\ \|v\|=1}}\min_{p\in\M_s}\|p(A)v\|
= \max_{\substack{v\in\R^n\\ \|v\|=1}}
\|\monic{A}{v}{v}\| \,.
\end{equation}
A unit norm vector $v\in\R^n$ and a corresponding monic polynomial $P_s(z;v)$
for which the value $\Phi_s(A)$ is attained are called a worst-case Arnoldi vector
and a worst-case Arnoldi polynomial for $A$ and $s$, respectively.
\end{definition}
\medskip{}

Since a worst-case vector or the corresponding worst-case polynomial for the GMRES method need not be unique in general (see \cite[Theorem~4.1]{FabLieTic13}), we expect the same to be true for the worst-case Arnoldi problem \eqref{eq:wcArnoldi}. Note that in general we have
\begin{equation}\label{eqn:idealArnoldi}
\Phi_s(A)=\max_{\substack{v\in\R^n\\ \|v\|=1}}
\min_{p\in\M_s}\|p(A)v\| \leq \min_{p\in\M_s} \max_{\substack{v\in\R^n\\ \|v\|=1}}\|p(A)v\|=
\min_{p\in\M_s} \|p(A)\|.
\end{equation}
The expression on the right hand side is called the \emph{ideal Arnoldi problem} for~$A$ and~$s$; see \cite{GreTre94}. If equality holds in \eqref{eqn:idealArnoldi}, then the worst-case Arnoldi polynomial in~\eqref{eq:wcArnoldi} is unique; see~\cite[Lemma~2.4]{TicLieFab07} and~\cite{LieTic09}.

The next result transfers~\cite[Theorem~2.2]{FabLieTic13} and the first part of~\cite[Theorem~2.5]{FabLieTic13} to the ACI($s$).
\medskip{}

\begin{theorem}\label{thm:monotone}
Let $A\in\mathbb{R}^{n\times n}$, let $s$ be an integer with $1\leq s<d(A)$, and let $v_{0}\in\mathbb{R}^{n}$ be such that $\|v_0\|=1$ and $d(A,v_{0})\geq s+1$. Then the vectors in \eqref{eqn:twk}--\eqref{eqn:vkp1} are all well defined, and
\begin{equation}\label{eqn:interlace}
\|\tw_k\|\leq\|\tv_{k+1}\|\leq\|\tw_{k+1}\|\leq\|\tv_{k+2}\|\leq \Phi_s(A^T)=\Phi_s(A),\quad k=0,1,2,\dots.
\end{equation}
Equality holds in the first inequality if and only if $v_{k}=\alpha v_{k+1}$ for some $\alpha\neq0$, and in the second if and only if $w_{k}=\beta w_{k+1}$ for some $\beta\neq0$.
\end{theorem}

\begin{proof}
We start by showing inductively that the vectors in \eqref{eqn:twk}--\eqref{eqn:vkp1} are all well defined and satisfy \eqref{eqn:interlace}. Suppose that for some $k\geq 0$ we have $\tw_{k}=\monic{A}{v_k}{v_k}\neq 0$ for some uniquely determined polynomial $\mpol{z}{v_k}\in\M_s$, so that $w_{k}=\tw_{k}/\|\tw_{k}\|$ is well defined. By construction we have $\tw_{k}\perp\mathcal{K}_{s}(A,v_{k})$, and therefore
\begin{align*}
\|\tw_{k}\|^{2} =\langle \monic{A}{v_k}{v_k},\tw_{k}\rangle=\langle A^{s}v_{k},\tw_{k}\rangle=
\langle q(A)v_{k},\tw_{k}\rangle
&=\langle v_{k},q(A^T)\tw_{k}\rangle,
\end{align*}
where $q\in\mathcal{M}_{s}$ is arbitrary. Using $w_{k}=\tw_{k}/\|\tw_{k}\|$ we obtain
\begin{align}\label{eqn:orth1}
\|\tw_{k}\| =
\langle v_{k},q(A^T)w_{k}\rangle.
\end{align}
Since $0<\|\tw_{k}\|$ and $q\in\mathcal{M}_{s}$ is arbitrary, we have $d(A^T,w_k)\geq s+1$. Using Lemma~\ref{lem:basic},
there is a unique polynomial $\mpol{z}{w_k} \in \M_s$ such that
$\tv_{k+1}=\monic{A^T}{w_k}{w_k} \neq 0$, and, using \eqref{eqn:orth1}
and choosing $q=\mpol{z}{w_k}$ yields
\begin{align}\label{eqn:ineq1}
\|\tw_{k}\| = \langle v_{k},\tv_{k+1}\rangle \leq\|\tv_{k+1}\|,
\end{align}
where we have used the Cauchy-Schwarz inequality and $\|v_{k}\|=1$.
Therefore, $v_{k+1} = \tv_{k+1}/\|\tv_{k+1}\|$ is well defined.

Similarly, we can prove that $\|\tv_{k+1}\|=\langle w_{k},q(A) v_{k+1}\rangle$
where $q\in\mathcal{M}_{s}$ is arbitrary.
Since $0<\|\tv_{k+1}\|$, we have $d(A,v_{k+1})\geq s+1$, and it holds that
\begin{align}\label{eqn:ineq1a}
\|\tv_{k+1}\|
 =\langle w_{k},\tw_{k+1}\rangle \leq \| \tw_{k+1} \|.
\end{align}
By induction, all vectors in  \eqref{eqn:twk}--\eqref{eqn:vkp1} are well defined
and satisfy \eqref{eqn:interlace}.

Moreover, since we have used the Cauchy-Schwarz inequality to obtain \eqref{eqn:ineq1}, we have $\|\tw_{k}\|\leq\|\tv_{k+1}\|$ with equality if and only if $v_{k}$ and $\tv_{k+1}$ are linearly dependent, or equivalently $v_{k}=\alpha v_{k+1}$ for some $\alpha\neq0$. Similarly, the equality in \eqref{eqn:ineq1a} holds
if and only if $w_k$ and $\tw_{k+1}$ are linearly dependent, or equivalently $w_k=\beta w_{k+1}$ for some $\beta\neq 0$.

It is obvious that $\| \tw_{k} \| \leq \Phi_s(A)$
and $\| \tv_{k} \| \leq \Phi_s(A^T)$ for each $k$. We only need to prove that
$\Phi_s(A) = \Phi_s(A^T)$.
Let $v$ be a unit norm worst-case Arnoldi vector for $A$ and $s$, and denote $\widetilde{w}=\monic{A}{v}{v}$ and $w\equiv \widetilde{w}/\|\widetilde{w}\|$,
so that $\|\widetilde{w}\| = \Phi_s(A)$. Then
\begin{align*}
\Phi_s(A)
= \left\langle \monic{A}{v}{v}, w \right\rangle
= \left\langle v, q(A^T) w \right\rangle
\leq
\min_{p\in \M_s} \| p(A^T) w \|
\leq
\Phi_s(A^T),
\end{align*}
where $q\in\mathcal{M}_{s}$ is arbitrary. Exchanging the roles of $A$ and $A^T$ shows the reverse inequality and completes the proof.
\end{proof}
\medskip{}

As shown in \eqref{eqn:interlace}, the two nondecreasing sequences $\{\|\tw_k\|\}$ and $\{\|\tv_k\|\}$ interlace each other and are both bounded from above by $\Phi_s(A)=\Phi_s(A^T)<\infty$. Thus, the sequences converge to the same limit,
\begin{align}\label{eqn:same-limit}
\|\tw_k\|\rightarrow\tau\quad\mbox{and}\quad \|\tv_k\|\rightarrow\tau\quad \mbox{for $k\rightarrow\infty$,}
\end{align}
where $\tau\leq \Phi_s(A)$.
This observation is helpful for proving the next result, which is the ACI($s$) version of the second part of~\cite[Theorem~2.5]{FabLieTic13}.
\medskip{}

\begin{theorem}\label{thm:evenodd}
Let $A\in\mathbb{R}^{n\times n}$, let $s$ be an integer with $1\leq s<d(A)$, and let $v_{0}\in\mathbb{R}^{n}$ be such that $\|v_0\|=1$ and $d(A,v_{0})\geq s+1$. Then the sequences of the normalized vectors in \eqref{eqn:twk}--\eqref{eqn:vkp1} satisfy
\[
\lim_{k\rightarrow\infty}\|w_{k+1}-w_{k}\|=0\quad\mbox{and}\quad\lim_{k\rightarrow\infty}\|v_{k+1}-v_{k}\|=0.
\]
\end{theorem}

\begin{proof}
Using similar ideas as in the proof of Theorem~\ref{thm:monotone}, for any $k\geq0$ we have
\begin{align*}
\frac{1}{2}\|v_{k+1}-v_{k}\|^{2} & =1-\langle v_{k+1},v_{k}\rangle\\
 & =1-\frac{1}{\|\tv_{k+1}\|}\langle \monic{A^T}{w_k}{w_k},v_k\rangle\\
 & =1-\frac{1}{\|\tv_{k+1}\|}\langle w_k,q(A) v_k\rangle\qquad\mbox{(for any \ensuremath{q\in\M_s})}\\
 & =1-\frac{1}{\|\tv_{k+1}\|}\langle w_k,\monic{A}{v_k}{v_k}\rangle\\
 & =1-\frac{1}{\|\tv_{k+1}\|}\langle w_k,\tw_k\rangle\\
 & =1-\frac{\|\tw_k\|}{\|\tv_{k+1}\|},
\end{align*}
where we have used $w_k\perp \K_s(A,v_k)$.
Since the sequences $\{\|\tv_k\|\}$ and $\{\|\tw_k\|\}$ converge to the same limit for $k\rightarrow\infty$, we have $\frac{\|\tw_k\|}{\|\tv_{k+1}\|}\rightarrow 1$ and hence
$\|v_{k+1}-v_{k}\|\rightarrow 0$ for $k\rightarrow\infty$. The proof for the sequence $\{w_k\}$ is analogous.
\end{proof}
\medskip{}

In the following we will focus on the properties of sequence $\{v_k\}$, with the understanding that analogous results can be formulated also for the sequence $\{w_k\}$. The unit norm vectors $v_{k}$ form a bounded sequence in ${\mathbb{R}}^{n}$. By the Bolzano-Weierstrass Theorem this sequence has a convergent subsequence, and thus it has limit vectors. The existence of a single limit vector of the sequence $\{y_{2k}\}$ in the iteration \eqref{eqn:yk}--\eqref{eqn:forsythe1} is the content of the original Forsythe conjecture (see Section~\ref{sec:forsythe}), and here we formulate the same conjecture for the ACI($s$):
\medskip{}

\textsc{ACI($s$) conjecture.} \emph{Let $A\in\mathbb{R}^{n\times n}$, let $s$ be an integer with $1\leq s<d(A)$, and let $v_{0}\in\mathbb{R}^{n}$ be such that $\|v_0\|=1$ and $d(A,v_{0})\geq s+1$. Then the sequence $\{v_{k}\}$ in \eqref{eqn:twk}--\eqref{eqn:vkp1} has a single limit vector.} \medskip{}

As shown in Theorem~\ref{thm:evenodd}, the (Euclidean) distance between consecutive vectors of the sequence $\{v_k\}$ shrinks to zero for $k\rightarrow\infty$. Because of this property it may be difficult to find a counterexample for the conjecture numerically. On the other hand, this property is not sufficient for the existence of a single limit vector. For example, the complex points $\mu_k=e^{\mathbf{i} s_k\pi}$, where $s_k=\sum_{j=1}^k 1/j$ for $k\geq 1$, satisfy
$$|\mu_k-\mu_{k+1}|=|1-e^{\mathbf{i} \pi/(k+1)}|\rightarrow 0\quad\mbox{for}\quad k\rightarrow\infty,$$
but the sequence $\{\mu_k\}$ does not converge, since $\{s_k\}$ diverges.

Using a similar notation as Forsythe in~\cite{For68}, we define for the given (possibly nonsymmetric) matrix $A\in\R^{n\times n}$ and integer $s$, where $1\leq s<d(A)$, the set
\[
\Sigma^{A}:=\{v\in\mathbb{R}^{n}\,:\,\|v\|=1\;\mbox{and}\;d(A,v)\geq s+1\},
\]
and the transformation
\begin{align}
T_{A}:\Sigma^{A}\rightarrow\Sigma^{A},\quad v\mapsto T_{A}(v):=\frac{\monic{A}{v}{v}}{\|\monic{A}{v}{v}\|}.
\end{align}
In the notation $\Sigma^A$ and $T_A$ we have suppressed the dependence on~$s$ for simplicity. However, for the analysis of the iteration \eqref{eqn:twk}--\eqref{eqn:vkp1}, it is important to explicitly consider the dependence on the matrix, since we operate with both $A$ and $A^T$.

In terms of \eqref{eqn:twk}--\eqref{eqn:vkp1}, if $v_0\in\Sigma^A$, then
$$v_{k+1}=T_{A^T}(T_{A}(v_{k}))\quad\mbox{for all $k\geq0$,}$$
where Theorem~\ref{thm:monotone} shows that the transformation $T_{A^T}\circ T_A:\Sigma^A\rightarrow\Sigma^A$ is indeed well defined. Moreover, both $T_{A^T}$ and $T_A$ are continuous, so that $T_{A^T}\circ T_{A}$ is also continuous.
The next result adapts~\cite[Theorems~3.8]{For68} to our context.
\medskip{}

\begin{theorem}\label{thm:limit-set}
Let $A\in\mathbb{R}^{n\times n}$, let $s$ be an integer with $1\leq s<d(A)$, and let $v_{0}\in\mathbb{R}^{n}$ be such that $\|v_0\|=1$ and $d(A,v_{0})\geq s+1$. Then the set
$\Sigma_{*}^{A}$ of limit vectors of the sequence $\{v_{k}\}$ in \eqref{eqn:twk}--\eqref{eqn:vkp1} satisfies:
\begin{itemize}
\item[(1)] $\Sigma_{*}^{A}$ is a closed and connected set in $\mathbb{R}^{n}$.
\item[(2)] $\Sigma_{*}^{A}\subseteq\Sigma^{A}$, and each $v_{*}\in\Sigma_{*}^{A}$
satisfies $v_{*}=T_{A^T}(T_{A}(v_{*}))$.
\end{itemize}
\end{theorem}

\begin{proof}
(1) Suppose that the sequence $\{v_{k}\}$ has two isolated limit vectors, $v_{*,1}$ and $v_{*,2}$. Then there exist neighborhoods $U_{1}$ and $U_{2}$ of $v_{*,1}$ and $v_{*,2}$, respectively, which do not contain other limit vectors of the sequence. Let $\{v_{k_{i}}\}$
be set of vectors in the sequence that is contained in $U_{1}$. Then the sequence $\{v_{k_{i}}\}$ converges to $v_{*,1}$, since this sequence can have only one limit vector. We know that the distance $\|v_{k+1}-v_{k}\|$ converges to zero for $k\rightarrow\infty$; see Theorem~\ref{thm:evenodd}. Therefore almost all (i.e., all except finitely many) successors of vectors in the set $\{v_{k_{i}}\}$ are contained in $U_{1}$. But this implies that only a finite number of the vectors of the sequence $\{v_{k}\}$ can be outside $U_{1}$, which contradicts that $v_{*,2}$ is a limit vector. Consequently, the sequence $\{v_{k}\}$ cannot have two isolated limit vectors.

(2) If $v_{*}\in\Sigma_{*}^{A}$, then there exists a subsequence $\{v_{k_{i}}\}$ of the sequence $\{v_{k}\}$ that converges to $v_{*}$. For each vector $v_{k_{i}}$ in the subsequence we have $T_{A^T}(T_{A}(v_{k_{i}}))=v_{k_{i}+1}$, and therefore
$$\|T_{A^T}(T_{A}(v_{k_{i}}))-v_{k_{i}}\|=\|v_{k_{i}+1}-v_{k_{i}}\|\rightarrow 0\quad \mbox{for $k\rightarrow\infty$};$$
see Theorem~\ref{thm:evenodd}. Since $T_{A^T}\circ T_{A}$ is continuous and $v_{k_{i}}\rightarrow v_{*}$,
we have $T_{A^T}(T_{A}(v_{k_{i}}))\rightarrow T_{A^T}(T_{A}(v_{*}))$, and hence $v_{*}=T_{A^T}(T_{A}(v_{*}))$.

It is clear that every $v_{*}\in\Sigma_{*}^{A}$ satisfies $\|v_{*}\|=1$. If $d(A,v_{*})\leq s$, then Lemma~\ref{lem:basic} implies $T_{A^T}(T_{A}(v_{*}))=0$, in contradiction to $T_{A^T}(T_{A}(v_{*}))=v_{*}$. Thus, $d(A,v_{*})\geq s+1$, which shows that $\Sigma_{*}^{A}\subseteq\Sigma^{A}$.
\end{proof}

\section{Results for special cases}\label{sec:special}

In this section we will first derive some general results about the ACI($s$) for symmetric matrices. Then we prove the ACI($1$) conjecture for symmetric, and for real orthogonal matrices with $d(A)=n$
and eigenvalues having only positive (or only negative) real parts.

\subsection{The ACI($s$) for symmetric matrices}\label{sec:4.1}

If $A=A^T\in\R^{n\times n}$, the steps \eqref{eqn:twk}--\eqref{eqn:wk} and \eqref{eqn:tvk}--\eqref{eqn:vkp1} in the ACI($s$) are identical, and hence we can write the algorithm in following simpler form:
\begin{align}
 & \mbox{For \ensuremath{k=0,1,2,\dots}}\nonumber \\
 & \hspace{1cm}\mbox{\ensuremath{\tv_{k+1}=\monic{A}{v_k}{v_k}},}\label{eqn:tvk-simple}\\
 & \hspace{1cm}\mbox{\ensuremath{v_{k+1}=\tv_{k+1}/\|\tv_{k+1}\|}.}\label{eqn:vk-simple}
\end{align}
Our conjecture now is that for each integer $s$ with $1\leq s<d(A)$, and unit norm initial vector $v_{0}\in\mathbb{R}^{n}$ with $d(A,v_{0})\geq s+1$, the sequence $\{v_{2k}\}$ in \eqref{eqn:tvk-simple}--\eqref{eqn:vk-simple} has a single limit vector.

Theorem~\ref{thm:monotone} for $A=A^T$ and the simplified algorithm \eqref{eqn:tvk-simple}--\eqref{eqn:vk-simple} says that
$$\|\tv_k\|\leq \|\tv_{k+1}\|,\quad k=0,1,2,\dots,$$
with equality if and only if $v_k=\alpha v_{k+2}$ for some $\alpha\neq 0$.
Knowing that $v_k$ and $v_{k+2}$ have unit norm, and that $\langle v_k,v_{k+2}\rangle > 0$ (see \eqref{eqn:ineq1}), we must have $\alpha =1$, i.e., $v_k=v_{k+2}$.
This can actually happen, as shown in the next result, which adapts~\cite[Theorem~4.8]{For68} to our context.

\medskip{}
\begin{theorem}
\label{thm:collinear} Let $A=A^{T}\in\mathbb{R}^{n\times n}$, let $s$ be an integer with $1\leq s<d(A)$, and let $v_{0}\in\mathbb{R}^{n}$ be such that $\|v_0\|=1$ and $d(A,v_{0})=s+1$. Then the vectors $v_{0},v_{2},v_{4},\dots$ in \eqref{eqn:tvk-simple}--\eqref{eqn:vk-simple} are all equal.
\end{theorem}

\begin{proof}
First note that if $d(A,v_0)=s+1$, then the subspace $\mathcal{K}_{s+1}(A,v_{0})$ is $A$-invariant. Therefore, all vectors $v_k$ are contained in the same $s+1$ dimensional space, and by Theorem \ref{thm:monotone} it holds that $d(A,v_{k})=s+1$ for all $k\geq 0$. As a consequence,
$$
	\dim \mathcal{K}_{s}(A,v_{k}) = s,\quad k = 0,1,\dots .
$$
It suffices to show that $v_{0}= v_{2}$, then the equality of all vectors $v_{0},v_{2},v_{4},\dots$ follows inductively.
By construction, $v_1\perp\mathcal{K}_{s}(A,v_{0})$ and $ v_2\perp\mathcal{K}_{s}(A,v_{1})$, where the second orthogonality condition is equivalent to
$v_1\perp\mathcal{K}_{s}(A,v_{2})$. Therefore,
$$
  v_1\;\perp\;\mathcal{K}_{s}(A,v_{0})\quad\mbox{and}\quad v_1\;\perp\;\mathcal{K}_{s}(A,v_{2}).
$$
Since $\dim \mathcal{K}_{s}(A,v_{0}) = \dim \mathcal{K}_{s}(A,v_{2}) = s$
and since
all vectors are contained in the same $s+1$ dimensional space,
we have $\mathcal{K}_{s}(A,v_{0})=\mathcal{K}_{s}(A,v_{2})$. Hence,
\[
v_{2}=\sum_{j=0}^{s-1}\gamma_{j}A^{j}v_{0},
\]
for some coefficients $\gamma_{0},\dots,\gamma_{s-1}\in\mathbb{R}$.
We will show that $\gamma_{0}=1$ and $\gamma_{j}=0$ for $j>0$.

Using $v_2 \perp A^i v_{1}$ for $i=0,\dots,s-1$ we obtain
\[
0=\langle A^i v_{1},v_{2}\rangle=
\langle  v_{1},A^i v_{2}\rangle =
\sum_{j=0}^{s-1}\gamma_{j}\langle v_{1},A^{j+i}v_{0}\rangle,\quad i=0,\dots,s-1.
\]
Now realize that
$\langle v_{1},A^{m}v_{0}\rangle= 0$ for $m=0,\dots,s-1$
and $\langle v_{1},A^{m}v_{0}\rangle \neq 0$ for $m\geq s$
because $A^m v_0 \in {\mathcal K}_{s+1}(A,v_0)\,\backslash\,{\mathcal K}_{s}(A,v_0)$.
In particular, for $i=1$ we obtain
\[
0=\langle A v_{1},v_{2}\rangle=\gamma_{s-1}\langle v_{1},A^{s}v_{0}\rangle.
\]
Since $\langle v_{1},A^{s}v_{0}\rangle\neq 0$, we get $\gamma_{s-1}=0$. Proceeding by induction we obtain $\gamma_{s-2}=\cdots=\gamma_{1}=0$. Therefore, $v_{2}=\gamma_{0}v_{0}$. Since $v_0$ and $v_{2}$ are unit norm vectors and $\langle v_0,v_{2}\rangle > 0$ (see \eqref{eqn:ineq1}) we get $\gamma_0=1$.
\end{proof}
\medskip{}

We next prove a result about the limit vectors of the sequence \eqref{eqn:tvk-simple}--\eqref{eqn:vk-simple}. The part about their grades adapts~\cite[Theorem~4.7]{For68} to our context.
\medskip{}

\begin{theorem}\label{thm:limit-degree}
Let $A=A^T\in\mathbb{R}^{n\times n}$, let $s$ be an integer with $1\leq s<d(A)$, and let $v_{0}\in\mathbb{R}^{n}$ be such that $\|v_0\|=1$ and $d(A,v_{0})\geq s+1$.
Then each limit vector $v_*$ of \eqref{eqn:tvk-simple}--\eqref{eqn:vk-simple} satisfies
\begin{equation}\label{eqn:interpolation}
(\Qpol{A}{v_*}-\tau^2 I)v_{*}=0,
\end{equation}
where $\tau$ is the limit value of the sequence $\{\|\widetilde{v}_k\|\}$  (as in~\eqref{eqn:same-limit}),
$$
 \Qpol{z}{v_*} := \mpol{z}{w_*}\mpol{z}{v_*}, \quad\mbox{and}\quad w_* := \frac{\monic{A}{v_*}{v_*}}{\tau},
$$
so that, in particular, $s<d(A,v_{*})\leq 2s$.
\end{theorem}
\smallskip

\begin{proof}
We use the notation of Theorem~\ref{thm:limit-set}. It is clear that each limit vector $v_*\in\Sigma^{A}_*$ satisfies $d(A,v_*)>s$, so we only need to show that $d(A,v_*)\leq 2s$. Moreover, since $A=A^T$, we have $T_A^2(v_{*})=v_{*}$.

By construction,
$$
T_A(v_*)=\frac{\monic{A}{v_*}{v_*}}{\tau}
$$ for a uniquely determined polynomial $\mpol{z}{v_*}\in\mathcal{M}_{s}$, and
$$
T_A^{2}(v_*)=\frac{\monic{A}{w_*}{w_*}}{\tau}, \quad w_*=T_A(v_*),
$$
for a uniquely determined polynomial $\mpol{z}{w_*}\in\mathcal{M}_{s}$; cf. Lemma~\ref{lem:basic}. Note that
$\tau$ is independent of the limit vector $v_*$. Thus,
\[
T_A^{2}(v_*)=\frac{\mpol{A}{w_*}\monic{A}{v_*}{v_*}}{{\tau^2}} =:
\frac{\monicQ{A}{v_*}{v_*}}{\tau^2}
\]
for a uniquely determined polynomial $\Qpol{z}{v_*}\in\mathcal{M}_{2s}$. Finally, using $T_A^{2}(v_{*})=v_{*}$ we obtain $(\Qpol{A}{v_*}-{\tau^2} I)v_{*}=0$, and hence $d(A,v_{*})\leq2s$.
\end{proof}
\medskip{}

Consider an orthogonal diagonalization of $A$, i.e., $A=U\Lambda U^T$, where $\Lambda={\rm diag}(\lambda_1,\dots,\lambda_n)$ and $U$ is orthogonal.
Let $v_* = U \nu$ where $\nu = [\nu_1,\dots,\nu_n]^T \in \mathbb{R}^n$, is the coordinate vector of $v_*$ in the eigenvector basis. Then the condition \eqref{eqn:interpolation} can also be written in the form
$$
	\Qpol{\lambda_i}{v_*}\nu_i = {\tau^2} \nu_i,\quad i=1,\dots,n.
$$
If $d(A,v_*)= d$, then there are exactly~$d$ nonzero coordinates $\nu_{k_i}$, $i=1,\dots,d$,
and the polynomial $\Qpol{z}{v_*}$ satisfies the~$d$ interpolation conditions
$$
	\Qpol{\lambda_{k_i}}{v_*}= \tau^2,\quad i=1,\dots,d.
$$
In particular, if $d=2s$, then $\Qpol{z}{v_*}$ is uniquely determined by the corresponding interpolation conditions, so that
$$
\Qpol{z}{v_*}  = \prod_{i=1}^{2s} (z-\lambda_{k_i}) + \tau^2.
$$
Since there exist only finitely many combinations of $2s$ distinct eigenvalues of~$A$, there are only finitely many polynomials $\Qpol{z}{v_*}$ that correspond to limit vectors~$v_*$ having degree $2s$.

The set $\Sigma_{*}^{A}$ of limit vectors of the sequence~$\{v_{2k}\}$, is a closed and connected set on the unit sphere; see Theorem~\ref{thm:limit-set}. Therefore, in order to show that $\Sigma_{*}^{A}$ consist of a single vector, it is sufficient to show that $\Sigma_{*}^{A}$ contains only finitely many vectors. The existence of only finitely many limit polynomials $\Qpol{z}{v_*}$ can potentially be used to prove that there can be only finitely many limit vectors $v_*$, thereby obtaining a proof of the ACI($s$) conjecture for symmetric matrices and general~$s$.

We point out that the convergence of the coefficients of the iteration polynomials and its consequences for the convergence of the sequence $\{v_{2k}\}$ is an essential ingredient in the work of Zhuk and Bondarenko~\cite{ZhuBon83} on the original Forsythe conjecture and the case $s=2$. In particular, they show that for $s=2$ the coefficients of the monic polynomials $\Qpol{z}{v_k}$ converge to their limit values. Quoting a paper by Zabolotskaya, they use as a proven fact that the convergence of the polynomial coefficients \emph{implies} the existence of a single limit vector of the even iterates; see~\cite[property~4, p.~429]{ZhuBon83}. However, the English translation of Zabolotskaya's paper~\cite[p.~238]{Zab79} states that the convergence of the polynomial coefficients would only \emph{indicate} the existence of a single limit vector. This may well be an imprecise translation of Zabolotskaya's Russian original, but it would nevertheless be very useful to have a more transparent proof of this essential property. Until then we consider the Forsythe conjecture for the case $s=2$ to be still open.

\subsection{Proof of the ACI($1$) conjecture for symmetric matrices}
\label{sec:s1}

If $s=1$, then Theorem~\ref{thm:limit-degree} shows that every limit vector $v_*$ of the sequence $\{v_{2k}\}$ in \eqref{eqn:tvk-simple}--\eqref{eqn:vk-simple} satisfies $d(A,v_*)=2$, i.e., $v_*$ is a linear combination of exactly two linearly independent eigenvectors of $A$. This observation is essential for proving the existence of a single limit vector of the sequence $\{v_{2k}\}$ for $s=1$.

Consider the ACI($1$) for $A=A^T\in\R^{n \times n}$, $s=1$. In this case we have (cf. \eqref{eqn:arnoldi})
\[
w=\monic{A}{v}{v}=Av-\alpha v\;\perp\;{\rm span}\{v\}.
\]
Thus, $0=v^{T}w=v^{T}Av-\alpha v^{T}v$, which yields
\[
\alpha=\frac{v^{T}Av}{v^{T}v},
\]
i.e., $\alpha$ is the \emph{Rayleigh quotient} of $A$ and $v$. Therefore, the algorithm \eqref{eqn:tvk-simple}--\eqref{eqn:vk-simple} can be written as follows:
\begin{align}
 & \mbox{For \ensuremath{k=0,1,2,\dots}}\nonumber \\
 & \hspace{1cm}\tv_{k+1}=(A-\rho_k I)v_{k},\quad \rho_{k}=v_{k}^{T}Av_{k},\label{eqn:s1-1}\\
 & \hspace{1cm}v_{k+1}=\tv_{k+1}/\|\tv_{k+1}\|.\label{eqn:s1-2}
 \end{align}
Let $A=U\Lambda U^{T}$ be an orthogonal diagonalization of $A$ with $\Lambda={\rm diag}(\lambda_{1},\dots,\lambda_{n})$. Then $\rho_k=(U^Tv_k)\Lambda (U^T v_k)$, and in \eqref{eqn:s1-1} we can write $U^T\tv_{k+1}=(\Lambda-\rho_{k}I)U^{T}v_{k}$. This shows that without loss of generality we can consider the behavior of the \eqref{eqn:s1-1}--\eqref{eqn:s1-2} for a diagonal matrix~$A$. Moreover, without loss of generality we will assume that $A$ has $n$ distinct eigenvalues.
\smallskip

\begin{theorem}\label{thm:s1}
Let $A={\rm diag}(\lambda_{1},\dots,\lambda_{n})\in\mathbb{R}^{n\times n}$
with $\lambda_{1}<\lambda_{2}<\cdots<\lambda_{n}$, and let $v_{0}$ be a unit norm initial vector with $d(A,v_0)\geq 2$.  Then the sequence $\{v_{2k}\}$ in  \eqref{eqn:s1-1}--\eqref{eqn:s1-2} converges to a single limit vector.
\end{theorem}
\smallskip

\begin{proof}
It is
sufficient to show that $\Sigma_{*}^{A}$ contains only finitely many vectors.
Let $v_*\in \Sigma_{*}^{A}$. We know that there exists a fixed $\tau>0$ (independent of~$v_*$) with $\|\tv_{k}\|=\|Av_k-(v_k^TAv_k)v_k\|\rightarrow\tau$ for $k\rightarrow\infty$; see \eqref{eqn:same-limit}. Thus,
\begin{equation}\label{eqn:tau-eq1}
\tau=\left\Vert Av_{*}-\left(v_{*}^{T}Av_{*}\right)v_{*}\right\Vert.
\end{equation}
Moreover, Theorem~\ref{thm:limit-degree} implies that $d(A,v_*)=2$, and hence $v_*=\alpha e_i+\beta e_j$ for some canonical basis vectors $e_{i}$ and $e_{j}$, $1\leq i< j\leq n$, and nonzero $\alpha,\beta\in\R$ that satisfy $\alpha^2+\beta^2=1$.

Now suppose that some vector $v=\alpha e_i+\beta e_j$ with $\alpha^2+\beta^2=1$ satisfies the equation \eqref{eqn:tau-eq1}. Then
\begin{align*}
\tau^{2} &=\left\Vert Av-\left(v^{T}Av\right)v\right\Vert ^{2}=
v^{T}A^{2}v-\left(v^{T}Av\right)^{2} \\ &=\alpha^{2}\lambda_{i}^{2}+\beta^{2}\lambda_{j}^{2}-\left(\alpha^{2}\lambda_{i}+\beta^{2}\lambda_{j}\right)^{2}\\
 & =  \alpha^{2}\lambda_{i}^{2}(1-\alpha^{2})+\beta^{2}\lambda_{j}^{2}(1-\beta^{2})-2\alpha^{2}\beta^{2}\lambda_{i}\lambda_{j}\\
 & =  \alpha^{2}\beta^{2}\left(\lambda_{i}^{2}-2\lambda_{i}\lambda_{j}+\lambda_{j}^{2}\right)\\
 & =  \alpha^{2}\beta^{2}\left(\lambda_{i}-\lambda_{j}\right)^{2}\\
 & =  \alpha^{2}\left(1-\alpha^{2}\right)\left(\lambda_{i}-\lambda_{j}\right)^{2}.
\end{align*}
There are only finitely many combinations of distinct $i,j\in\{1,2,\dots,n\}$, and for each such combination there are only finitely many values of $\alpha\in\R$ that satisfy
$$\alpha^{2}\left(1-\alpha^{2}\right)=\frac{\tau^2}{\left(\lambda_{i}-\lambda_{j}\right)^{2}}.$$
Consequently, there are only finitely many vectors of the form $v=\alpha e_i+\beta e_j$ with $\alpha^2+\beta^2=1$ that satisfy the equation \eqref{eqn:tau-eq1}. Therefore there can be only finitely many $v_*\in \Sigma_{*}^{A}$, which shows that the sequence $\{v_{2k}\}$ converges to a single limit vector.
\end{proof}
\medskip

Afanasjew, Eiermann, Ernst, and G\"uttel~\cite[Section~3]{AfaEieErnGue08} also have shown that the sequence $\{v_{2k}\}$ in  \eqref{eqn:s1-1}--\eqref{eqn:s1-2} converges to a single limit vector. Similar to the original proof of Forsythe~\cite{For68}, their proof is based on first showing that each limit vector is a linear combination of $e_1$ and $e_n$ (or eigenvectors corresponding to the smallest and largest eigenvalue of~$A$), and then showing that there can be only finitely many such combinations. Variations of this approach have appeared also in other proofs for symmetric positive definite matrices and $s=1$; see, e.g.,~\cite{Aka59,GonSch16}. The approach is longer and more technical than our proof of Theorem~\ref{thm:s1}, but the results give more information about the structure of the limiting vectors.

\subsection{On the ACI($1$) conjecture for orthogonal matrices}\label{sec:ACI-orth}

In this section we first prove the ACI($1$) conjecture for orthogonal matrices $A\in\R^{n\times n}$ with $0\notin F(A)$, the field of values of~$A$, and we then comment on the behavior when $0\in F(A)$.

Let $A\in\R^{n\times n}$ be orthogonal. Starting from a unit norm vector $v_0\in\R^n$ with $d(A,v_0)\geq s+1$, the ACI($1$) is as follows:
\begin{align}
& \mbox{For $k=0,1,2,\dots$}\nonumber\\
& \hspace{1cm}\widetilde{w}_k = (A-\alpha_k I)v_k,\quad \alpha_k=v_{k}^TAv_k,\label{eqn:ACI1a}\\
& \hspace{1cm}w_k = \widetilde{w}_k / \|\widetilde{w}_k\|,\\
& \hspace{1cm}\widetilde{v}_{k+1} = (A^T-\beta_k I)w_k,\quad \beta_k=w_{k}^TA^Tw_k=w_{k}^TAw_k,\\
& \hspace{1cm}v_{k+1} =\widetilde{v}_k/\|\widetilde{v}_k\|.\label{eqn:ACI1b}
\end{align}
The Rayleigh quotients $\alpha_k$ and $\beta_k$ in the ACI($1$) are real elements in $F(A)$. Moreover,
\begin{align}
\|\widetilde{w}_k\|^2 &= v_{k}^T (A^T-\alpha_k I)(A-\alpha_k I)v_k=1-\alpha_k^2,
\label{eqn:norw}
\\
\|\widetilde{v}_{k+1}\|^2 &= w_{k}^T (A-\beta_k I)(A^T-\beta_k I)w_k=1-\beta_k^2,
\label{eqn:norv}
\end{align}
where we have used that $A$ is orthogonal. By construction, $\|\widetilde{w}_k\|>0$ and $\|\widetilde{v}_{k+1}\|>0$, so that $\alpha_k,\beta_k\in (-1,1)$.
Since $\|\widetilde{w}_k\|\leq \|\widetilde{v}_{k+1}\|$, we always have
\begin{equation}\label{eqn:betalpha}
|\beta_k| \leq |\alpha_k|\,.
\end{equation}
We know from \eqref{eqn:same-limit} that the sequences $\{\|\widetilde{w}_k\|\}$ and $\{\|\widetilde{v}_k\|\}$ converge to the same limit. Therefore, $\alpha_k^2-\beta_k^2\rightarrow 0$, so that there exist $\alpha,\beta\in (-1,1)$ with
\begin{equation}\label{eqn:alpha-lim}
|\alpha|=|\beta|,\quad \alpha_k\rightarrow\alpha,\quad\mbox{and}\quad \beta_k\rightarrow\beta.
\end{equation}
Here $\alpha$ and $\beta$ are independent of the limit vectors of the sequences $\{v_k\}$ and $\{w_k\}$.

It is well known that an orthogonal matrix~$A\in \R^{n\times n}$ can be orthogonally block-diagonalized as $A=UGU^T$ with
\begin{equation*}
U=[U_1,\dots,U_m,u_1,\dots,u_k],\quad G={\rm diag}(G_1,\dots,G_m,[\pm 1],\dots,[\pm 1]),
\end{equation*}
where $U^TU=I$, $U_1,\dots,U_m\in\R^{n\times 2}$, $u_1,\dots u_k\in\R^n$, and
\begin{equation}\label{eqn:Dj}
G_j=\begin{bmatrix} c_j & s_j\\ -s_j & c_j\end{bmatrix}\in\R^{2\times 2},
\quad c_j^2+s_j^2=1,\quad s_j\neq 0,\quad j=1,\dots,m.
\end{equation}
The blocks $G_j$ in \eqref{eqn:Dj} correspond to the non-real eigenvalues of~$A$, which occur in complex conjugate pairs $c_j\pm \mathbf{i}\, s_j$ with real parts $c_j\in (-1,1)$. By transforming the iterates of the ACI($1$) with $U^T$ (similarly to the transformation for symmetric matrices in Section~\ref{sec:s1}) we can assume without loss of generality that~$A$ is in the block-diagonal form, i.e., that $A={\rm diag}(G_1,\dots,G_m,[\pm 1],\dots,[\pm 1])$.

We will now assume that~$A\in\R^{n\times n}$ is orthogonal with $d(A)=n$ and $0\notin F(A)$. Then in the orthogonal block-diagonalization of $A$ there can be at most one block of size $1\times 1$, either $[1]$ or $[-1]$, and $c_1,\dots,c_m$ are pairwise distinct and either all positive or all negative. For simplicity, we will state and prove the next results only for the positive case.
\smallskip

\begin{lemma}\label{lem:orthogonal1}
Let $A={\rm diag}(G_1,\dots,G_m,[1])\in\R^{n\times n}$, where possibly the block $[1]$ does not occur, have blocks $G_j$ as in \eqref{eqn:Dj} with $0<c_1<\cdots<c_m$. If $v_0\in\R^n$ is any unit norm initial vector with $d(A,v_0)\geq 2$, then any limit vector $v_*\in\Sigma_*^A$ of the sequence $\{v_k\}$ in \eqref{eqn:ACI1a}--\eqref{eqn:ACI1b} satisfies $d(A,v_*)=2$, and $\alpha=\beta$ is equal to the real part of an eigenvalue of~$A$.
\end{lemma}
\smallskip

\begin{proof}
Let $v_*$ be any limit vector of the sequence $\{v_k\}$. It is clear that $d(A,v_*)\geq 2$, so it suffices to show that $d(A,v_*)\leq 2$.

According to Theorem~\ref{thm:limit-set} we have $v_{*}=T_{A^T}(T_{A}(v_{*}))$, i.e.,
\begin{align*}
& \widetilde{w}_* = (A-\alpha I)v_*,\quad \alpha=v_{*}^TAv_*,\\
& w_* = \widetilde{w}_* / \|\widetilde{w}_*\|,\\
& \widetilde{v}_{*} = (A^T-\beta I)w_*,\quad \beta=w_{*}^TA^Tw_*=w_{*}^TAw_*,\\
& v_{*} =\widetilde{v}_*/\|\widetilde{v}_*\|.
\end{align*}
Moreover, \eqref{eqn:norw}--\eqref{eqn:norv} and \eqref{eqn:alpha-lim} yield $\|\widetilde{v}_*\|^2=\|\widetilde{w}_*\|^2=1-\alpha^2$, so that
\begin{align*}
v_* &=T_{A^T}(T_{A}(v_{*}))=
\frac{(A^T-\beta I)(A-\alpha I)v_*}{\|\widetilde{w}_*\|\, \|\widetilde{v}_*\|}=\frac{((1+\alpha\beta)I-\alpha A^T-\beta A)v_*}{1-\alpha^2}.
\end{align*}
This yields the equation
\begin{equation}\label{eqn:v-eq}
(\alpha A^T+\beta A)v_*=(\alpha^2+\alpha\beta)v_*.
\end{equation}
Since $\alpha$ and $\beta$ are real elements in~$F(A)$ with $|\alpha|=|\beta|$, we must have $\alpha=\beta\in(0,1)$. Then \eqref{eqn:v-eq} implies that the limit vector $v_*$ satisfies
\begin{equation}\label{eqn:v-eq-1}
\frac12 (A+A^T)v_*=\alpha v_*.
\end{equation}
Note that
$$\frac12 (A+A^T)={\diag}(c_1 I_2,\dots, c_m I_2, [1]),$$
where possibly the block $[1]$ does not occur. Since $c_1,\dots,c_m\in (0,1)$ are pairwise distinct, we must have $\alpha=c_j$ for exactly one index $j\in\{1,\dots,m\}$. Consequently, $v_*=[0,\dots,0,y^T,0,\dots,0]^T$ for some unit norm vector $y\in\R^2$ corresponding to the $j$th block, which shows that $d(A,v_*)\leq 2$.
\end{proof}

\medskip
For the given orthogonal matrix $A={\rm diag}(G_1,\dots,G_m,[1])\in\R^{n\times n}$, where possibly the block $[1]$ does not occur, we will consider the corresponding block-partitioning of the vectors $v_k$, i.e.,
\begin{equation}\label{eqn:partitioning}
v_k = \begin{bmatrix}
v_k^{(1)}\\
\vdots\\
v_k^{(m)}\\
v_k^{(m+1)}
\end{bmatrix},\quad  v_k^{(j)}\in \R^2, \quad j=1,\dots,m,\quad v_k^{(m+1)} \in \R,
\end{equation}
where possibly the block $v_k^{(m+1)}$ does not occur.
\medskip

\begin{lemma}\label{lem:c1}
Let $A={\rm diag}(G_1,\dots,G_m,[1])\in\R^{n\times n}$, where possibly the block $[1]$ does not occur, have blocks $G_j$ as in \eqref{eqn:Dj} with $0<c_1<\cdots<c_m$. If
$v_0\in\R^n$ is any unit norm initial vector with $d(A,v_0)\geq 2$ and $v_0^{(1)} \neq 0$, then any limit vector $v_*\in\Sigma_*^A$ of the sequence $\{v_k\}$ in \eqref{eqn:ACI1a}--\eqref{eqn:ACI1b} has zero entries except for the block $v_*^{(1)}$, $\| v_*^{(1)}\|=1$, and $\alpha=\beta=c_1$.
\end{lemma}
\smallskip

\begin{proof}
The ACI($1$) yields
\begin{align*}
v_{k+1}=\frac{(A^{T}-\beta_{k}I)(A-\alpha_{k}I)}{\|\widetilde{w}_{k}\|\,\|\widetilde{v}_{k+1}\|}\,v_{k},
\quad k=0,1,2,\dots.
\end{align*}
If the block $[1]$ occurs, then
$$v_{k+1}^{(m+1)}=\frac{(1-\beta_k)(1-\alpha_k)}{\sqrt{1-\beta_{k}^{2}}\sqrt{1-\alpha_{k}^{2}}}v_k^{(m+1)}
=\left(\frac{(1-\beta_k)(1-\alpha_k)}{(1+\beta_k)(1+\alpha_k)}\right)^{1/2}v_k^{(m+1)},\ \
k=0,1,2,\dots,$$
where we have used \eqref{eqn:norw}--\eqref{eqn:norv}. Since the factor that multiplies $v_k^{(m+1)}$ is less than~$1$, we have $v_k^{(m+1)}\rightarrow 0$ and hence $v_*^{(m+1)}=0$.

The assertion is trivial if $m=1$, so we can assume that $m>1$. For the other blocks we have the equation
\begin{equation}\label{eqn:blocks}
v_{k+1}^{(j)}=
\frac{(G_{j}^{T}-\beta_{k}I)(G_{j}-\alpha_{k}I)}{\sqrt{1-\beta_{k}^{2}}\sqrt{1-\alpha_{k}^{2}}}\,v_{k}^{(j)},
\quad j=1,\dots,m.
\end{equation}
Taking the squared norm in \eqref{eqn:blocks} and using the fact that
\[
(G_{j}^{T}-\delta I)(G_{j}-\delta I)=(1+\delta^{2}-2c_{j}\delta)I
\]
holds for any real $\delta$, we obtain
\begin{align*}
\|v_{k+1}^{(j)}\|^{2} & =  \left(\frac{1+\beta_{k}^{2}-2\beta_{k}c_{j}}{1-\beta_{k}^{2}}\right)\left(\frac{1+\alpha_{k}^{2}-2\alpha_{k}c_{j}}{1-\alpha_{k}^{2}}\right)\|v_{k}^{(j)}\|^{2}\\
 & = \left(1+2\beta_{k}\frac{\beta_{k}-c_{j}}{1-\beta_{k}^{2}}\right)\left(1+2\alpha_{k}\frac{\alpha_{k}-c_{j}}{1-\alpha_{k}^{2}}\right)\|v_{k}^{(j)}\|^{2}.
\end{align*}

We will now prove by contradiction that
$\beta_{k}\rightarrow c_{1}$ and $\alpha_{k}\rightarrow c_{1}$.
Suppose that $\alpha_{k}\rightarrow c_{\ell}$ and $\beta_{k}\rightarrow c_{\ell}$ for some $\ell>1$. Using \eqref{eqn:betalpha} we know that
$$
\alpha_{k}\geq\beta_{k} \geq \alpha_{k+1} \geq \beta_{k+1} \geq \dots \geq c_{\ell},
$$
and hence, for $j=1$,
\[
\|v_{k+1}^{(1)}\|\geq
\left( 1+2c_{\ell}\frac{c_{\ell}-c_{1}}{1-c_{\ell}^{2}} \right)
\|v_{k}^{(1)}\|\geq
\left( 1+2c_{\ell}\frac{c_{\ell}-c_{1}}{1-c_{\ell}^{2}} \right)^{k+1}
\|v_{0}^{(1)}\|.
\]
But since $1+2c_{\ell}\frac{c_{\ell}-c_{1}}{1-c_{\ell}^{2}}>1$ and $v_0^{(1)}\neq 0$, this implies $\|v_{k}^{(1)}\|\rightarrow\infty$, in contradiction to the normalization of the vectors $v_k$. Therefore, $\beta_{k}\rightarrow c_{1}$ and $\alpha_{k}\rightarrow c_{1}$, and Lemma~\ref{lem:orthogonal1} yields that $v_*$ has the required form.
\end{proof}
\medskip{}

Our next goal is show that under the assumptions of Lemma~\ref{lem:c1} there is only one uniquely determined vector in $\Sigma_*^A$.
In the following lemma we show that norms of the blocks $v_{k}^{(j)}$ for $j>1$ converge to zero at least linearly.
\smallskip

\begin{lemma}\label{lem:unit0}
If $A$ and $v_0$ are as in Lemma~\ref{lem:c1}, then exist an index $k_{0}$
and $0<\varrho<1$, such that for all $k\geq k_{0}$,
\[
\|v_{k+1}^{(j)}\|\leq \varrho \,\|v_{k}^{(j)}\|,\quad j=2,\dots,m+1.
\]
\end{lemma}

\begin{proof}
We know from the proof of Lemma~\ref{lem:c1} that
\begin{equation*}
\|v_{k+1}^{(j)}\|^{2}  \,= \, \zeta_{k}^{(j)} \|v_{k}^{(j)}\|^{2},
\quad j=1,\dots,m+1,
\end{equation*}
where
\begin{equation}\label{eqn:zetaj}
\zeta_{k}^{(j)}
=
\left(1+2\beta_{k}\frac{\beta_{k}-c_{j}}{1-\beta_{k}^{2}}\right)\left(1+2\alpha_{k}\frac{\alpha_{k}-c_{j}}{1-\alpha_{k}^{2}}\right),\quad j=1,\dots,m,
\end{equation}
and
\begin{equation}\label{eqn:zetam}
\zeta_{k}^{(m+1)} = \frac{(1-\beta_k)(1-\alpha_k)}{(1+\beta_k)(1+\alpha_k)}.
\end{equation}
Using that $\alpha_{k}\geq\beta_{k}\geq\alpha_{k+1}\geq\beta_{k+1}\geq\dots\geq c_{1}$ we obtain
$$
   \zeta_k^{(j)} \rightarrow \left(1+2c_{1}\frac{c_{1}-c_{j}}{1-c_{1}^{2}}\right)^2
   \;\; j=1,\dots,m,
\quad\mbox{and}\quad
    \zeta_k^{(m+1)} \rightarrow \left(\frac{1-c_1}{1+c_1}\right)^2,
$$
where the limit values are both strictly less than 1.
Consequently there exist $k_{0}\geq 0$ and $0<\varrho<1$ such that
\[
\|v_{k+1}^{(j)}\|^{2}\leq\varrho^{2}\|v_{k}^{(j)}\|^{2},\quad j=2,\dots,m+1,
\]
for all $k\geq k_{0}$.
\end{proof}\smallskip

Since we are interested only in convergence of the sequence of vectors, we can assume for simplicity and without loss of generality that $k_{0}=0$ in Lemma~\ref{lem:unit0}. For all $k\geq 0$ we then have
\begin{eqnarray*}
1 & = & \|v_{k+1}\|^{2}=\|v_{k+1}^{(1)}\|^{2}+\sum_{j=2}^{m+1}\|v_{k+1}^{(j)}\|^{2}
\leq \|v_{k+1}^{(1)}\|^{2}+\varrho^{2}\sum_{j=2}^{m+1}\|v_{k}^{(j)}\|^{2}\\
 & \leq & \|v_{k+1}^{(1)}\|^{2}+\varrho^{2(k+1)}\sum_{j=2}^{m+1}\|v_{0}^{(j)}\|^{2}.
\end{eqnarray*}
Therefore,
\begin{equation}\label{eqn:key}
1-\|v_{k+1}^{(1)}\|^{2}\leq\varrho^{2(k+1)}.
\end{equation}
The next result shows that the sequence $\{v_{k}^{(1)}\}\subset {\mathbb R}^2$ converges.
\smallskip

\begin{lemma} \label{lem:unit1}
Let $A$ and $v_0$ be as in Lemma~\ref{lem:c1}, and assume without loss of generality that $k_0=0$ in Lemma~\ref{lem:unit0}. Then, for all $k\geq 0$,
\[
\|v_{k+1}^{(1)}-v_{k}^{(1)}\|\leq \varrho^{k}.
\]
\end{lemma}

\begin{proof}
We know that
\[
\|v_{k+1}^{(1)}-v_{k}^{(1)}\|^{2}=\|v_{k+1}^{(1)}\|^{2}+\|v_{k}^{(1)}\|^{2}
-2\left(v_{k+1}^{(1)}\right)^{T}\left(v_{k}^{(1)}\right),
\]
where
\begin{eqnarray*}
\left(v_{k+1}^{(1)}\right)^{T}\left(v_{k}^{(1)}\right) & = & \left(v_{k}^{(1)}\right)^{T}\frac{(G_{1}^{T}-\beta_{k}I)(G_{1}-\alpha_{k}I)}{\sqrt{1-\alpha_{k}^{2}}\sqrt{1-\beta_{k}^{2}}}\,v_{k}^{(1)}\\
 & = & \frac{1-c_{1}\left(\alpha_{k}+\beta_{k}\right)+\alpha_{k}\beta_{k}}{\sqrt{1-\alpha_{k}^{2}}\sqrt{1-\beta_{k}^{2}}}\left\Vert v_{k}^{(1)}\right\Vert ^{2}\\
 & \geq &
 \frac{1-\beta_k\left(\alpha_{k}+\beta_{k}\right)+\alpha_{k}\beta_{k}}{\sqrt{1-\alpha_{k}^{2}}\sqrt{1-\beta_{k}^{2}}}\left\Vert v_{k}^{(1)}\right\Vert ^{2} \geq
\left\Vert v_{k}^{(1)}\right\Vert ^{2} .
\end{eqnarray*}
Note that the next to last inequality follows from $\beta_k \geq c_1 > 0$. Therefore
\begin{eqnarray*}
\|v_{k+1}^{(1)}-v_{k}^{(1)}\|^{2}&\leq&\|v_{k+1}^{(1)}\|^{2}-\|v_{k}^{(1)}\|^{2}
\leq 1-\|v_{k}^{(1)}\|^{2} \leq \varrho^{2k},
\end{eqnarray*}
where we have used \eqref{eqn:key}.
\end{proof}
\smallskip

Now we are ready to prove the convergence theorem.

\smallskip
\begin{theorem} \label{thm:convergence}
Let $A={\rm diag}(G_{1},\dots,G_{m},[1])\in\mathbb{R}^{n\times n}$, $m\geq1$, be an orthogonal matrix with blocks $G_{j}$ as in \eqref{eqn:Dj}, and $0<c_{1}<c_{2}<\cdots<c_{m}<1$, where possibly the block $[1]$ does not occur. Let $v_{0}\in\mathbb{R}^{n}$ be any unit norm initial vector with $d(A,v_{0})\geq2$ such that $v_{0}^{(1)}\neq0$. Then the sequence $\{v_{k}\}$ in \eqref{eqn:ACI1a}--\eqref{eqn:ACI1b} converges to a single limit vector.
\end{theorem}

\begin{proof}
Using Lemmas~\ref{lem:unit0} and \ref{lem:unit1}, and assuming without loss of generality that $k_0=0$ in Lemma~\ref{lem:unit0}, we obtain for all $k\geq 0$,
\begin{eqnarray*}
\|v_{k+1}-v_{k}\|^{2}	&=&	\|v_{k+1}^{(1)}-v_{k}^{(1)}\|^{2}+\sum_{j=2}^{m+1}\|v_{k+1}^{(j)}-v_{k}^{(j)}\|^{2}\\
	&\leq& \varrho^{2k}+2\varrho^{2k}\sum_{j=2}^{m+1}\|v_{0}^{(j)}\|^{2}
	\,\leq\, 3\varrho^{2k}
\end{eqnarray*}	
which implies $\sum_{k=0}^{\infty}\|v_{k+1}-v_{k}\|<\infty$, and hence
finishes the proof.
\end{proof}
\smallskip

We now consider an orthogonal matrix $A\in\R^{n\times n}$ with $d(A)=n$ and $0\in F(A)$. Analogously to Lemma~\ref{lem:c1} we will show that (with appropriate assumptions on the initial vector~$v_0$) the sequences of the Rayleigh quotients $\{\alpha_k\}$ and $\{\beta_k\}$ in \eqref{eqn:ACI1a}--\eqref{eqn:ACI1b} converge to $\min_{z\in F(A)} \mathrm{Re}(z)$. Under the assumption $0\in F(A)$ this means that $\alpha=\beta=0$ in \eqref{eqn:alpha-lim}.

We can again assume without loss of generality that $A$ is block diagonal,
\[
A={\rm diag}([-1],G_{1},\dots,G_{m},[1])\in\mathbb{R}^{n\times n},\quad m\geq 1,
\]
with blocks $G_{j}$ as in \eqref{eqn:Dj}, and $-1<c_{1}<c_{2}<\cdots<c_{m}<1$, where the blocks $[-1]$ or $[1]$ possibly do not occur. Let us also formally define $c_0 = -1$ and $c_{m+1}=1$. We will consider
a block-partitioning of the vectors $v_{k}$ as in \eqref{eqn:partitioning}, where we add the block $v_{k}^{(0)}$ if the block $[-1]$ occurs.
\smallskip

\begin{lemma} \label{lem:zeroFA}
Let $A={\rm diag}([-1],G_{1},\dots,G_{m},[1])\in\mathbb{R}^{n\times n}$,
$m\geq1$, be an orthogonal matrix with blocks $G_{j}$ as in \eqref{eqn:Dj},
and $-1=c_0<c_{1}<c_{2}<\cdots<c_{m}<1=c_{m+1}$, where possibly the blocks $[-1]$
or $[1]$ do not occur. Let $v_{0}\in\mathbb{R}^{n}$ be a unit norm initial vector with $d(A,v_{0})\geq2$. Suppose that $v_{0}^{(\ell)}\neq0$ for some $\ell$,  $1\leq \ell\leq m$,  and that $v_{0}^{(j)}\neq0$ for some $j$, $0\leq j\leq m+1$, such that $c_{\ell}c_j \leq 0$. Then $\alpha=\beta=0$ in \eqref{eqn:alpha-lim}.
\end{lemma}
\smallskip

\begin{proof}
From \eqref{eqn:betalpha} and \eqref{eqn:alpha-lim} we know that $|\alpha_{k}|\geq|\beta_{k}|$ and that $|\alpha_{k}|\rightarrow\gamma$ and $|\beta_{k}|\rightarrow\gamma$. Since $d(A,v_{0})\geq2$ and
$v_{0}^{(\ell)}\neq0$ for some $1\leq \ell\leq m$, we have $|\alpha_{0}|<1$, and therefore $0\leq\gamma<1$. We will prove by contradiction that $\gamma=0$. Let $\gamma>0$, i.e., $\alpha_{k}\rightarrow\alpha$ and $\beta_{k}\rightarrow\beta$, where $1>|\alpha|=|\beta|=\gamma>0$.

Suppose first that $\alpha=-\beta\neq0$, and consider a block
for which $v_{0}^{(\ell)}\neq0$, $1\leq \ell\leq m$. Then (see \eqref{eqn:zetaj})
\begin{eqnarray*}
\zeta_k^{(\ell)} \rightarrow\left(1+2\alpha\frac{\alpha+c_{j}}{1-\alpha{}^{2}}\right)\left(1+2\alpha\frac{\alpha-c_{j}}{1-\alpha^{2}}\right) & = & 1+\frac{4\alpha^{2}}{1-\alpha{}^{2}}\left(\frac{1-c_{j}^{2}}{1-\alpha{}^{2}}\right)>1,
\end{eqnarray*}
Therefore $\|v_{k}^{(\ell)}\|\rightarrow\infty$, which is a contradiction.

Suppose now that $1>\alpha=\beta>0$, and consider a block
with $c_j\leq 0$ such that $v_{0}^{(j)}\neq0$, $0\leq j \leq m$.
If $j=0$, then
\[
\zeta_k^{(0)} \rightarrow \frac{(1+\alpha)^{2}(1+\alpha)^{2}}{\left(1-\alpha^{2}\right)\left(1-\alpha^{2}\right)}>1,
\]
and if $j>0$, then
\[
\zeta_k^{(j)} \rightarrow \left(1+2\frac{\alpha^{2}-\alpha c_{j}}{1-\alpha^{2}}\right)^{2}>1
\]
giving a contradiction. Note that if $-1<\beta=\alpha<0$, then we derive a contradiction by considering
a block with $c_j\geq 0$ such that $v_{0}^{(j)}\neq0$.
\end{proof}
\smallskip

In order to prove that the sequence $\{v_{k}\}$ converges to a single limit vector also
when $0\in F(A)$, it would be sufficient to show that there exists $0<\varrho<1$ such that
\[
\ensuremath{\|v_{k+1}-v_{k}\|=\mathcal{O}\left(\varrho^{k}\right)}
\]
for $k$ sufficiently large; cf. the proof of Theorem~\ref{thm:convergence}. With some effort we can show that there is a constant $C>0$ such that
\[
\|v_{k+1}-v_{k}\|\leq C |\alpha_{k}|
\]
for $k$ sufficiently large. It remains to prove that the coefficients $\alpha_{k}$ converge to zero linearly, but this steps needs further investigation.

We will now show by an example that for an orthogonal matrix $A$ with $0\in F(A)$ a limit vector $v_*$ of the sequence $\{v_{k}\}$ in \eqref{eqn:ACI1a}--\eqref{eqn:ACI1b} may satisfy $d(A,v_*)>2$.  This is a significant difference to the case $0\notin F(A)$, where $d(A,v_*)=2$ holds for any limit vector $v_*$; see Lemma~\ref{lem:orthogonal1}.

\smallskip
\begin{example}
Let $A={\rm diag}(G_1,G_2)\in\R^{4\times 4}$ with blocks $G_{j}$ as in \eqref{eqn:Dj}, where $c_1\in (-1,0)$ and $c_2=-c_1$, so that in particular $0\in F(A)$. Let $v_0\in \R^4$ be any unit norm vector with $v_0^{(1)}=v_0^{(2)}$. Then $d(A)=d(A,v_0)=4$, and in \eqref{eqn:ACI1a}--\eqref{eqn:ACI1b} we get $\alpha_0=v_0^TAv_0=0$, and hence
\begin{align*}
\widetilde{w}_0 &=Av_0,\quad \|\widetilde{w}_0\|=1,\quad w_0=\widetilde{w}_0=Av_0,\quad \beta_0 =w_0^TAw_0=v_0^TAv_0=0,\\
\widetilde{v}_1&=A^Tw_0=v_0,\quad v_1=v_0.
\end{align*}
Consequently, $v_k=v_0$ and $w_k=Av_0$ for all $k\geq 0$. The unique limit point of the sequence $\{v_k\}$ is given by $v_*=v_0$, and $d(A,v_*)=4$.

\end{example}
\medskip

The observation in this example that a limit vector $v_*$ satisfies $d(A,v_*)=n$ is not surprising if we look at the proof of Lemma~\ref{lem:zeroFA}: Since $\alpha=0$, we have
\[
\zeta_{k}^{(j)}\rightarrow 1,\quad j=0,1,\dots,m+1,
\]
which would indicate that $v_*^{(j)}\neq 0$ if $v_{0}^{(j)}\neq 0$, and even $d(A,v_*)=d(A,v_0)$. However, we did not prove that this holds in general.

Finally, recall that the ideal Arnoldi problem for a matrix $A\in\R^{n\times n}$ and $s=1$ is given by
$$\min_{\alpha\in\R}\|A-\alpha I\|;$$
see \eqref{eqn:idealArnoldi}. A straightforward computation shows that
$$\min_{\alpha\in\R}\|A-\alpha I\|=\max_{\substack{v\in\R^n\\ \|v\|=1}} (\|Av\|^2-\langle v,Av\rangle^2)^{1/2}.$$
If $A\in\R^{n\times n}$ is normal and has the eigenvalues $\lambda_1,\dots,\lambda_n\in {\mathbb C}$, then
$$\min_{\alpha\in\R}\|A-\alpha I\|=\min_{\alpha\in\R}\max_{1\leq j\leq n} |\lambda_j-\alpha|,$$
and the unique solution of this problem is given by the center of the closed disk of smallest radius in the complex plane that contains $\lambda_1,\dots,\lambda_n$; see~\cite[Section~2.4]{FabLieTic10}.

If $A$ and $v_0$ are as in Lemma~\ref{lem:c1}, then $c_1$ is the center of that disk, and for the unique limit vector $v_*$ we have
$$\|Av_*\|^2-\langle v_*,Av_*\rangle^2=\|G_1v_*^{(1)}\|^2-\langle v_*^{(1)},G_1v_*^{(1)}\rangle^2=1-c_1^2,$$
so that
$$\min_{\alpha\in\R}\|A-\alpha I\|=\|A-c_1 I\|=\|Av_*-c_1v_*\|=(1-c_1^2)^{1/2}.$$

On the other hand, if $A$ and $v_0$ are as in Lemma~\ref{lem:zeroFA}, then $0$ is the center of that disk. Any $v_*\in\Sigma_*^A$ then satisfies $v_*^TAv_*=0$, which gives $\|Av_*\|^2-\langle v_*,Av_*\rangle^2=1$, and $\min_{\alpha\in\R}\|A-\alpha I\|=\|A\|=\|Av_*\|=1$.

In short, any $v_*\in\Sigma_*^A$ solves the ideal Arnoldi problem for an orthogonal matrix~$A\in\R^{n\times n}$ and $s=1$ (when $v_0$ satisfies the appropriate assumptions.) Note that this observation gives no insight into the uniqueness of the limit vectors.

\section{Concluding remarks}\label{sec:Conclusions}
In this paper we have introduced and analyzed the ACI($s$) in order to generalize the Forsythe conjecture from symmetric positive definite matrices to symmetric and nonsymmetric matrices. We were able to prove the existence of a single limit vector of the sequence $\{v_{2k}\}$ in \eqref{eqn:s1-1}--\eqref{eqn:s1-2} (the case $s=1$ and $A=A^T$), and the sequence $\{v_k\}$ in \eqref{eqn:ACI1a}--\eqref{eqn:ACI1b} (the case $s=1$ and orthogonal~$A$) when $0\notin F(A)$. Our uniqueness proof for $s=1$ and $A=A^T$ is much simpler than other previously published proofs.

In the case $s=1$, the property of monotonically increasing norms in Theorem~\ref{thm:evenodd}, and its proof based on orthogonality and the Cauchy-Schwarz inequality, appear to the be closely related to the property of monotonically decreasing residual norms in the Rayleigh quotient iteration (for symmetric matrices) as well as Ostrowski's two-sided iteration and Parlett's two-sided iteration (for general matrices). The monotonic residual norms ultimately yield the global convergence of these iterations; see, e.g.,~\cite{Par74}. Working out the exact relations between the iterations studied in the context of the Forsythe conjecture and the Rayleigh quotient based iterations remains a subject of further research.

Also, the Forsythe conjecture and its generalization to the ACI($s$) still remain largely open.

\section*{Acknowledgements} We thank G\'erard Meurant for helpful comments.

\bibliographystyle{siam}
\bibliography{forsythe}

\end{document}